\documentclass[a4paper]{amsart}

\usepackage[english]{babel}
\usepackage{verbatim}
\usepackage{lmodern}
\usepackage{textcomp}
\usepackage{color}
\usepackage{amssymb,mathrsfs,enumerate}
\usepackage{cases}
\usepackage{dsfont}

\usepackage{color}
\definecolor{citation}{rgb}{0.2,0.58,0.2} 
\definecolor{formula}{rgb}{0.1,0.2,0.6}
\definecolor{url}{rgb}{0.3,0,0.5} 
\usepackage[colorlinks=true,linkcolor=formula,urlcolor=url,citecolor=citation]{hyperref}

\def\vs{\vspace{0.5mm}}
\def\dd{d_{\textrm{o}}}
\def\dxy{{\rm d}\xi{\rm d}\eta}

\def\eps{\varepsilon}
\def\h{\mathds{H}}
\def\dxi{{\rm d}\xi}
\def\dxieta{\,{\rm d}\xi{\rm d}\eta}

\def\p{{\phi}}

\allowdisplaybreaks
\makeatletter
\DeclareRobustCommand*{\bfseries}{%
	\not@math@alphabet\bfseries\mathbf
	\fontseries\bfdefault\selectfont
	\boldmath
}

\DeclareMathOperator*{\osc}{osc}
\DeclareMathOperator*{\essinf}{ess \, \inf}
\DeclareMathOperator*{\esssup}{ess \, \sup}

\makeatother

\newlength{\defbaselineskip}
\newcommand{\setlinespacing}[1]
{\setlength{\baselineskip}{#1 \defbaselineskip}}

\addtocounter{page}{1}
\newtheorem{thm}{Theorem}[section]
\newtheorem{prop}[thm]{Proposition}
\newtheorem{corol}[thm]{Corollary}
\newtheorem{lemma}[thm]{Lemma}
\newtheorem{defn}[thm]{Definition}
\newtheorem{rem}[thm]{Remark}

\numberwithin{equation}{section}
\allowdisplaybreaks[4]

\newcommand{\rr}{\varrho}
\newcommand{\snr}[1]{\lvert #1\rvert}
\newcommand{\norma}[1]{\left\| #1\right\|}

\newcommand\ap{``}
\def \er {{\mathds{R}}}
\def \ph{{\phi}}
\def \Lc{{\mathcal{L}}}
\def \He{{\mathds{H}^n}}
\def \Tail{\textup{Tail}}
\def \Tl{\textup{Tail}}
\def\Xint#1{\mathchoice
	{\XXint\displaystyle\textstyle{#1}}%
	{\XXint\textstyle\scriptstyle{#1}}%
	{\XXint\scriptstyle\scriptscriptstyle{#1}}%
	{\XXint\scriptscriptstyle\scriptscriptstyle{#1}} %
	\!\int}
\def\XXint#1#2#3{{\setbox0=\hbox{$#1{#2#3}{\int}$}
		\vcenter{\hbox{$#2#3$}}\kern-.5\wd0}}

\def\dashint{\Xint-}

\title[Nonlinear and nonlocal equations in~$\He$]{Nonlinear fractional equations \\ in the Heisenberg group}

\author[G.~Palatucci]{Giampiero Palatucci}  \address{Giampiero Palatucci\\Dipartimento di Scienze Matematiche, Fisiche e Informatiche, Universit\`a di Parma\\ Parco Area delle Scienze 53/a, Campus, 43124 Parma, Italy} \email{\url{giampiero.palatucci@unipr.it}}
\author[M.~Piccinini]{Mirco Piccinini}  \address{Mirco Piccinini\\Dipartimento di Scienze Matematiche, Fisiche e Informatiche, Universit\`a di Parma\\ Parco Area delle Scienze 53/a, Campus, 43124 Parma, Italy} \email{\url{mirco.piccinini@unipr.it}}
\begin{document}
	\vskip -8in
	\begin{center}
		\rule{8.5cm}{0.5pt}\\[-0.1cm] 
		{\small To appear in 
			{\it Bruno Pini Math. Anal. Semin.},
			2024.
		}
		\\[-0.25cm] \rule{8.5cm}{0.5pt}
	\end{center}
	\vspace {2.2cm}

	    %
	    %
	    
		\begin{abstract}
		We deal with a wide class of nonlinear nonlocal equations led by integro-differential operators of order~$(s,p)$, with summability exponent~$p \in (1,\infty)$ and differentiability exponent~$s\in (0,1)$, whose prototype is the fractional subLaplacian in the Heisenberg group. We present very recent boundedness and regularity estimates (up to the boundary) for the involved weak solutions, and we introduce  the nonlocal counterpart of the Perron Method in the Heisenberg group, by recalling some results on the fractional obstacle problem. Throughout the paper we also list various related open problems.
	\end{abstract}

	\subjclass[2010]{Primary 35H20, 35R05;
		Secondary 35B05, 35B30, 35D10, 35B45, 47G20\vspace{1mm}} 	
	\keywords{nonlocal operators, fractional subLaplacian, De~Giorgi-Nash-Moser theory, Heisenberg group, Caccioppoli estimates, obstacle problems, Perron's method\vspace{1mm}}

   \thanks{{\it Aknowledgements}. The authors are also supported by INdAM Project \ap Fenomeni non locali in problemi locali", CUP\_E55F22000270001.
   	The second author is also supported by the Project \ap Local vs Nonlocal: mixed type operators and nonuniform ellipticity", CUP\_D91B21005370003.}
	
	\maketitle

 \setcounter{equation}{0}\setcounter{thm}{0}
	\setcounter{tocdepth}{2} 
{\footnotesize 	\setlinespacing{0.79}
	\vspace{2mm}
	\centering
	
	\tableofcontents 
}

\setlinespacing{1.07}

	%
\section{Introduction}
	In this notes we deal with a  wide class of nonlinear and nonlocal integro-differential operators, defined on  suitable Sobolev spaces, whose explicit expression is given by
	\begin{equation}\label{main_op}
		\Lc u (\xi) \equiv \Lc_{s,p} u (\xi) := P.~\!V. \int_{\He} \frac{|u(\xi)-u(\eta)|^{p-2}(u(\xi)-u(\eta))}{\dd(\eta^{-1}\circ\xi)^{Q+sp}} \, {\rm d}\eta, \qquad \xi \in \He,
	\end{equation}
    where~$p \in (1,\infty)$,~$s \in (0,1)$,~$Q\equiv 2n+2$ is the homogeneous dimension and~$\dd$ is a homogeneous norm on~$\He$; see Section~\ref{sec_prel} below for the precise definitions and further details.
	
    \vspace{2mm}
    We firstly focus on the linear case when $p=2$. Integro-differential operators as in the form in~\eqref{main_op} do arise as a generalization of the classical {\it conformally invariant} fractional subLaplacian~$(-\Delta_{\He})^{(s)}$ on~$\He$ which was first introduced in~\cite{BFM13} via the spectral formula
     $$
     (-\Delta_{\He})^{s} :=2^s \snr{T}^s \frac{{\bf \Gamma}\big(-\frac{1}{2}\Delta_{\He}\snr{T}^{-1} +\frac{1+s}{2}\big)}{{\bf \Gamma}\big(-\frac{1}{2}\Delta_{\He}\snr{T}^{-1} +\frac{1-s}{2}\big)}, \qquad s \in (0,1),
     $$
     with~${\bf \Gamma}(x):=\int_0^\infty t^{x-1}e^{-t}\, {\rm d}t$ being the Euler Gamma function,~$\Delta_{\He}$ being the classical Kohn-Spencer subLaplacian in~$\He$, and~$T=\partial_t$ being the vertical vector field.
     
     As proven in~\cite[Proposition~4.1]{RT16} the following representation formula holds true
     \begin{equation}\label{frac_sublap}
	 (-\Delta_{\He})^{s} u(\xi)\, :=\, C(n,s) \,P.~\!V. \int_{\He}\frac{u(\xi)-u(\eta)}{|\eta^{-1}\circ \xi|_{\He}^{Q+2s}}\,{\rm d}\eta, \qquad \xi \in \He,
   \end{equation}
  where~$C(n,s)$ is a positive constant which depends only on~$n$ and~$s$.

   Notice that the operator defined in Formula~\eqref{frac_sublap} above does not coincide, for any~$s \in (0,1)$, with the classical power of the subLaplacian~$-\Delta_{\He}^s$ defined via heat kernel
   \begin{equation}\label{not_conf_inv}
   -\Delta_{\He}^s u (\xi) := -\frac{s}{{\bf \Gamma}(1-s)}\int_0^\infty \frac{1}{t^{1+s}}\big(P_t u(\xi)-u(\xi)\big) \,{\rm d}t, \qquad u \in C^\infty_0(\He),
    \end{equation}
   with~$P_tu(\xi):= e^{-\Delta_{\He} t}u(\xi)$ being the heat semigroup built by Folland in~\cite{Fol73}.

    \vspace{2mm}   
    In the last decades, many relevant results have been achieved in the study of fractional operators in the Heisenberg group both in the form~\eqref{not_conf_inv} and~\eqref{frac_sublap}.  As an example, we just recall the paper~\cite{RT16}, where several Hardy inequalities were established for the conformally invariant fractional subLaplacian~\eqref{frac_sublap}. We also refer the reader  to~\cite{CCR15} for related Hardy and uncertainty inequalities on general stratified Lie groups involving fractional powers of the subLaplacian, and to~\cite{AM18}, where Sobolev and Morrey-type embeddings are derived for fractional order Sobolev spaces. 
    Moreover, very important results have been obtained based on the construction of fractional operators via a Caffarelli-Silvestre-type extension  as seen in the Euclidean framework in~\cite{CS07}. In particular, we mention the recent series of papers~\cite{GT21,GT21b} where the authors used  ad-hoc Dirichlet-to-Neumann map to built explicit fundamental solutions to~\eqref{frac_sublap} and~\eqref{not_conf_inv} and to prove some intertwining formulas for~$(-\Delta_{\He})^{s}$ and~$-\Delta_{\He}^s$. Also, a Liouville-type theorem was proven in~\cite{CT16}; Harnack and H\"older results in Carnot groups in~\cite{FF15}; whereas the connection with the fractional perimeters of sets in Carnot group can be found in~\cite{FMPPS18}.
    We also refer to~\cite{OT23} for mixed local/nonlocal operators on the Heisenberg group where the nonlocal part is given by~ the standard subLaplacian operator~$(-\Delta_{\He})^{s}$ defined in~\eqref{frac_sublap}.
    
    \vspace{4mm} 
	In the present notes we consider the more general situation when a $p$-growth exponent is involved. Nonlinear operators on stratified Lie group both in their nonlocal and local formulation are connected with several concrete models arising from many different contexts where the analysis in subRiemannian geometry revealed to be decisive; see for example~\cite{BGM18,CMS10,OS22,PS99,Wey50} and the references therein.
	Here, we present some very recent results regarding the De Giorgi-Nash-Moser theory for the weak solutions to general equations led by integro-differential operators~$\Lc$ defined in~\eqref{main_op}, and we conclude by recalling a Perron-type resolutivity theorem for the related Dirichlet problem. All the forthcoming theorems are proven  in~\cite{MPPP23,PP22,Pic22}, and they are the natural extension in the Heisenberg setting of the Euclidean counterpart proven for a class of very general $p$-fractional Laplacian-type operators in~\cite{DKP14,DKP16,KKP16,KKP17,LL17}. For further results on equations where the leading operator is such as the one in~\eqref{main_op}, we also refer to~\cite{HS23,KS18,KS20,WD20}.
	
	\vspace{2mm}
	The main difficulty in dealing with the operator~$\Lc$ in~\eqref{main_op} lies in its own definition. Indeed, we have to deal both with its nonlocal and nonlinear structure and with the underlying nonEuclidean geometry of the Heisenberg group. Moreover, since the integrability exponent~$p$ can differ from 2,  most of the classical tools successfully extended in the Heisenberg framework in the linear case such as the Dirichlet-to-Neumann map in~\cite{FGMT15,GT21,GT21b}, the approach via (non-commutative) Fourier representation, and many others, are not trivially applicable. As natural when dealing with nonlocal operators we have to consider a tail-type contribution which gives a precise control on the long-range interactions that naturally occur.
	
	\begin{defn}[{\bf Nonlocal tail}; Definition~(1.5) in~{\rm\cite{MPPP23,PP22}}]
	 For any~$p \in (1,\infty)$ and~$s \in (0,1)$ we call {\rm nonlocal tail} of a function~$u$ in a ball~$B_R(\xi_0) \subset \He$ of radius~$R>0$ and centred in~$\xi_0\in \He$ the quantity defined as follows
	\begin{equation}\label{tail}
		\Tail(u;\xi_0,R):= \left(R^{sp} \int_{\He \smallsetminus B_R(\xi_0)}|u(\xi)|^{p-1}|\xi_0^{-1} \circ \xi|_{\He}^{-Q -sp}\,{\rm d}\xi\right)^{\frac{1}{p-1}}\!.
	\end{equation}
	\end{defn}
    In the Euclidean framework the nonlocal tail has been firstly introduced in the study of nonlocal operators in~\cite{DKP14,DKP16} and then subsequently used to derive fine properties of solutions to nonlocal equations; see e.~\!g.~\cite{BLS18,KKP16,KKP17,KMS15} and the references therein. We also refer the interested reader to the survey article~\cite{Pal18}.
    Several extension of the nonlocal tail are available in the parabolic setting~\cite{Strom19,DZZ21,Liao22,APT22} as well as in the kinetic framework~\cite{AP23}. In the considered Heisenberg framework, the nonlocal tail has been firstly introduced in the papers~\cite{PP22,MPPP23}, and -- despite its young age -- it has been already proven to be a fundamental tool in order to achieve several results, as for instance the ones regarding the regularity for the obstacle problem and the Perron method (\cite{Pic22}), and those involving the fractional critical (Folland-Stein)-Sobolev embedding (see~\cite{GLV23}, and also~\cite{PPT23,PP23,PPT23b}).

    \vspace{2mm}
    In clear accordance with the definition of the tail function in~\eqref{tail}, we can consider as done in~\cite{KKP16,KKP17} the associated {\it tail space}~$L^{p-1}_{sp}$ given by
    $$
    L^{p-1}_{sp}(\He):=\left\{ v\in L^{p-1}_{\rm loc}(\He): \int_{\He} \frac{|v(\xi)|^{p-1}}{(1+|\xi|_{\He})^{Q+sp}}\, {\rm d}\xi < \infty\right\}.
    $$ 
    The previous definitions are very natural when dealing with operators of the type~\eqref{main_op}; we refer the reader to Section~\ref{sec_prel} below for further details.
    
    \vspace{2mm}
    \subsection{Outline of the paper}
    In Section~\ref{sec_prel} we recall some useful facts about the Heisenberg group and its related fractional Sobolev spaces. In Section~\ref{sec_hold} we present the De Giorgi-Nash-Moser Theory for~$\Lc$. Section~\ref{sec_obs} is devoted to the related obstacle problem, and to extend up to the boundary the boundedness and H\"older results presented in Section~\ref{sec_hold}. 
    In Section~\ref{sec_perron} we solve the Dirichlet problem related to~$\Lc$ in the sense of Perron in general open bounded set. We conclude with Section~\ref{sec_further} giving some indications about further developments of the results presented in the present notes.

    \vspace{2mm}
	%
	%
	\section{Preliminaries}\label{sec_prel}
	We fix some notation which will be used throughout the rest of the paper. Firstly, we will denote by~$c$ a general positive constant greater than~$1$ which will not necessarily be the same at different occurrences and which can also change from line to line. 

	\subsection{The Heisenberg group}
	We start by introducing some definitions and results about the Heisenberg group. For further details we refer to the book~\cite{BLU07} by Bonfiglioli, Lanconelli and Uguzzoni.
	
	\vspace{2mm}
	The Heisenberg group~$\He$ is the Lie group which has~$\er^n_x\times \er^n_y \times \er_t$ as background manifold and whose group law~$\circ$ and dilations~$ \delta_\lambda$ are given by
     $$
		\xi \circ \xi' := \big(x+x',\, y+y',\, t+t'+2\langle y,x'\rangle-2\langle x,y'\rangle \big), \qquad \forall \xi=(x,y,t), \xi'= (x',y',t') \in \He
	$$
     and 
     $$
         \delta_\lambda(\xi):=\big(\lambda x, \, \lambda y, \,\lambda^2 t\big),
 	 $$
	 respectively.
 	 It can be checked that the identity element is the origin~$0$,  the inverse~$\xi^{-1} = - \xi$. The Jacobian of~$\delta_\lambda$ is equal to~$\lambda^{2n+2}$, and the number~$Q \equiv 2n+2$ is usually called {\it homogeneous dimension} of~$\He$.
	
	The Jacobian basis of the Heisenberg Lie algebra~$\mathfrak{h}^n$ of~$\He$ is given by the following vector fields
	$$
	X_j := \partial_{x_j} +2y_j \partial_t, \quad
	X_{n+j}:= \partial_{y_j}-2x_j \partial_t, \quad 1 \leq j \leq n, \quad
	T := \partial_t.
	$$
     Moreover,  a simple computation shows that
	$$
	[X_j,X_{n+j}]:= X_j X_{n+j}-X_{n+j}X_j = -4\partial_t, \quad \text{for every}~1 \leq j \leq n,
	$$
	so that
	$$
    \textup{rank}\Big(\textup{Lie}\{X_1,\dots,X_{2n}\}\Big)\ = \ 2n+1.
	$$
	The group $\He$ is thus a {\it Carnot group} with the following stratification~$\mathfrak{h}^n$ of the Lie algebra,
	$$
	\mathfrak{h}^n = \textup{span}\{X_1,\dots,X_{2n}\}\oplus \textup{span}\{T\}.
	$$
    Given a domain~$\Omega \subset \He$, for any~$u\in C^1(\Omega;\,\er)$ we define the {\it horizontal gradient}~$\nabla_{\He} u$  of~$u$ by
	$$
	\nabla_{\He} u (\xi):= \Big(X_1u(\xi),\dots, X_{2n}u(\xi)\Big).
    $$    
	\begin{defn} 
		A \textup{homogeneous norm} on~$\He$ is a continuous function {\rm (}with respect to the Euclidean topology\,{\rm )}~${d_{\rm o}} : \He \rightarrow [0,+\infty)$ such that:
		\begin{enumerate}[\rm(i)]
			\item{
				${d_{\rm o}}(\delta_\lambda(\xi))=\lambda {d_{\rm o}}(\xi)$, for every~$\lambda>0$ and every~$\xi \in \He$;
			}\vspace{1mm}
			\item{
				${d_{\rm o}}(\xi)=0$ if and only if~$\xi=0$.
			}
		\end{enumerate}
		Moreover, we say that the homogeneous norm~${d_{\rm o}}$ is {\rm symmetric} if~$
		{d_{\rm o}}(\xi^{-1})={d_{\rm o}}(\xi)$, for any~$\xi \in \He$.
	\end{defn}
	Fixed an homogeneous norm~$d_{\rm o}$ on~$\He$, the function~$\Psi$ defined on the set of all pairs of elements of~$\He$ by
		$$
		\Psi(\xi,\eta):={d_{\rm o}} (\eta^{-1}\circ \xi),
		$$
		is a pseudometric on $\He$.
		
		\vspace{2mm}
		Homogenous norms are not in general proper norms on~$\He$. However, for any homogeneous norm~$d_{\rm o}$ on~$\He$ we have that there exists a constant~$\Lambda>0$ such that
		$$
	   \Lambda^{-1}	\snr{\xi}_{\He} \le d_{\rm o}(\xi) \le \Lambda \snr{\xi}_{\He}, \quad \forall \xi \in \He,~\snr{\xi}_{\He}:= \left((\snr{x}^2 +\snr{y}^2)^2+t^2\right)^\frac{1}{4}.
		$$
        The function~$\snr{\cdot}_{\He}$ in the display above is {\it the Kor\'anyi distance} in the Heisenberg group, and it is actually a norm; for the proof we refer to~\cite{Cyg81} and to Example~5.1 in~\cite{BFS17}.
        
       	For any fixed~$\xi_0 \in \He$ and~$R>0$, we denote by~$B_R(\xi_0)$ the ball with center~$\xi_0$ and radius~$R$, given by
     	$$
     	B_R(\xi_0):=\Big\{\xi \in \He : |\xi_0^{-1}\circ \xi|_{\He} < R\Big\}.
	    $$
    \vspace{2mm}
	\subsection{The fractional Sobolev spaces}
    We recall now some definitions and a few related basic results about our fractional functional setting. For further details, we refer the reader to~\cite{AM18,KS18}; see also~\cite{DPV12} 
     for 
    the basics of the fractional Sobolev spaces in the Euclidean setting.
	\vspace{1mm}

	Let~$p \in (1,\infty)$,~$s \in (0,1)$, and let~$u : \He \rightarrow \er$ be a measurable function. 
	The fractional Sobolev spaces~$HW^{s,p}$ on the Heisenberg group is defined by
    $$
		HW^{s,p}(\He):=\Big\{u \in L^p(\He):  \frac{\snr{u(\xi)-u(\eta)}}{\snr{\eta^{-1}\circ \xi}^{s+\frac{Q}{p}}} \in L^p(\He \times \He)\Big\},
    $$
	endowed with the natural fractional norm
    $$
			\|u\|_{HW^{s,p}(\He)}:= \Big(\int_{\He}\snr{u}^p \, {\rm d}\xi +\underbrace{\int_{\He}\int_{\He}\frac{\snr{u(\xi)-u(\eta)}^p}{\snr{\eta^{-1}\circ \xi}^{Q+sp}}\, {\rm d}\xi {\rm d}\eta}_{=:[u]_{HW^{s,p}(\He)}}\Big)^\frac{1}{p}, \qquad u \in HW^{s,p}(\He).
    $$
     In a similar fashion, given a domain~$\Omega \subset \He$, one can define the  fractional Sobolev space~$HW^{s,p}(\Omega)$. By~$HW^{s,p}_0(\Omega)$ we denote the closure of~$C_0^\infty(\Omega)$ in~$HW^{s,p}(\He)$. Conversely, if~$v \in HW^{s,p}(\Omega')$ with~ $\Omega \Subset \Omega'$ and~$v=0$ outside of~$\Omega$ a.~\!e., then~$v$ has a representative in~$HW_0^{s,p}(\Omega)$ as well.
     
	As expected, one can prove that classical embedding theorems still hold in our fractional functional setting.  In particular, we refer to the result in~\cite{KS18}, where the authors prove the following
	\begin{thm}[Theorem~2.5 in~\cite{KS18}]\label{sobolev}
		Let~$p \in (1,\infty)$ and~$s \in (0,1)$ such that~$sp<Q$. For any measurable compactly supported function $u : \He \rightarrow \er$ there exists a positive constant~$c\equiv c(n,p,s)$ such that
		\begin{equation*}
			\|u\|^p_{L^{p^*}(\He)}\, \leq \, c \,[u]^p_{HW^{s,p}(\He)}\,,
		\end{equation*}
		where~$p^* = {Qp}/{(Q-sp)}$ is the critical Sobolev exponent.
	\end{thm}
    For an analogous fractional Sobolev embedding on stratified Lie groups we refer to~\cite[Theorem~1]{GKR23}, where the authors also investigate some relevant properties of nonlinear fractional eigenvalue problems on stratified Lie groups. Moreover, for further results on fractional Sobolev embedding on stratified Lie groups we also refer to~\cite{GKR23-2}.

   In the event that~$sp >Q$ we have a classical Morrey-type embedding.
   \begin{thm}[Theorem~1.5 in \cite{AM18}]
   Let~$p \in (1,\infty)$,~$s\in(0,1)$ such that~$sp>Q$. For any~$u \in HW^{s,p}(\He)$ we have that
   $$
   \|u\|_{C^{0,\beta}(\He)} \le c \|u\|_{HW^{s,p}(\He)},
   $$
   where~$c \equiv c(n,s,p)>0$ and~$\beta = (sp-Q)/p$.
   \end{thm}
   Moreover, by applying the same strategy developed in~\cite[Section~2]{AKM18}, one can prove a fractional Poincar\'e-type inequality.
    \begin{prop} 
    	Let~$p \geq 1$ and~$s \in (0,1)$ and~$u \in HW^{s,p}_{\rm loc}(\Omega)$. Then, for any~$B_r(\xi_0) \Subset \Omega$ we have that
    	\begin{equation*}
    		\int_{B_r(\xi_0)}|u-(u)_r|^p \, {\rm d}\xi \leq c \, r^{sp} \int_{B_r(\xi_0)}\int_{ B_r(\xi_0)} \frac{|u(\xi)-u(\eta)|^p}{|\eta^{-1}\circ \xi|_{\He}^{Q+sp}} \,{\rm d}\xi {\rm d}\eta,
    	\end{equation*}
    	where~$c\equiv c(n,p)>0$ and~$(u)_r := \displaystyle\dashint_{B_r(\xi_0)}u\,{\rm d}\xi$.
    \end{prop}

	\vspace{2mm}
	We conclude this section by introducing the  Dirichlet problem related to~$\Lc$ in~\eqref{main_op}, as well as recalling the natural definition of weak solutions.
	
	Let~$\Omega$ be a bounded open set in $\He$ and $g \in HW^{s,p}(\He)$, we consider the Dirichlet problem
	\begin{equation}\label{pbm}
		\begin{cases}
			\Lc u   = f(\cdot,u) & \text{in} \ \Omega,\\[0.5ex]
			u   =g & \text{in} \ \He \smallsetminus\Omega,
		\end{cases}
	\end{equation}
	where~$\Lc$ is given by~\eqref{main_op} and the datum~$u\mapsto f \equiv f(\cdot, u)$ is bounded locally uniformly for $\xi\in~\Omega$.
	
     For any~$g \in HW^{s,p}(\He)$ consider the classes~$\mathcal{K}^{\pm}_g (\Omega)$ defined by
	$$
	\mathcal{K}^{\pm}_g (\Omega):=\Big\{  v \in HW^{s,p}(\He): (g-v)_\pm \in HW^{s,p}_0(\Omega)\Big\},
	$$
	and 
	$$
	\mathcal{K}_g(\Omega) := \mathcal{K}^+_g(\Omega) \cap \mathcal{K}^-_g(\Omega) = \Big\{  v \in HW^{s,p}(\He): v-g \in HW^{s,p}_0(\Omega)\Big\}.
	$$
	We have the following
	\begin{defn}[{\bf Fractional weak solutions}. See Section~2.2 in~{\rm\cite{MPPP23,PP22}}]\label{solution to inhomo pbm}
	\mbox{}\\	A function~$u \in \mathcal{K}^-_g(\Omega)$  {\rm (}$\mathcal{K}^+_g(\Omega)$, respectively{\rm )} is a \textup{weak subsolution} {\rm (}\textup{supersolution}, resp.{\rm)} to~\eqref{pbm} if 
		\begin{eqnarray*}
			&& \int_{\He}\int_{\He}\frac{|u(\xi)-u(\eta)|^{p-2}(u(\xi)-u(\eta))(\psi(\xi)-\psi(\eta))}{\dd(\eta^{-1}\circ \xi)^{Q+sp}}  \,{\rm d}\xi \,{\rm d}\eta \\*[0.5ex]
			&& \hspace{6cm} \leq \big(\geq,\textrm{resp.}\big) \int_{\He}f(\xi,u(\xi))\psi(\xi) \, \,{\rm d}\xi,
		\end{eqnarray*}
		for any nonnegative~$ \psi \in HW_0^{s,p}(\Omega)$.
		\\	A function~$u \in \mathcal{K}_g(\Omega)$ is a \textup{weak solution} to~\eqref{pbm} if it is both a weak sub- and supersolution. 
	\end{defn}
    Few remarks are in order. Firstly, the requirement that~$u \in HW^{s,p}(\He)$ can be weakened assuming~$u \in HW^{s,p}_{\rm loc}(\Omega) \cap L^{p-1}_{sp}(\He)$. Moreover, in case of weak supersolution it makes no difference requiring that~$u \in HW^{s,p}_{\rm loc}(\Omega)$ with~$u_- \in L^{p-1}_{sp}(\He)$, as shown by the following lemma, which will also be one of the keypoints in the proof of the Harnack inequalities in forthcoming Section~\ref{sec_h}.
    
      \begin{lemma}[{\bf Tail estimate}. See~Lemma~2.9 in~{\rm\cite{Pic22}}]
  \mbox{}\\  	Let~$u$ be a weak supersolution in~$B_{2r}(\xi_0)$. Then, there exists~$c \equiv c(n,p,s, \Lambda)>0$ such that
    	\begin{eqnarray*}
    		\Tail(u_+;\xi_0,r) &\le& c\,r^\frac{sp-1-Q}{p-1}[u]_{HW^{h,p-1}(B_r(\xi_0))}+c\,r^{-\frac{Q}{p-1}}\|u\|_{L^{p-1}(B_r(\xi_0))}\\*
    		&& +\, c\,\Tail(u_-;\xi_0,r)+ c\,r^\frac{sp}{p-1} \|f\|_{L^\infty(B_r(\xi_0))}^\frac{1}{p-1},
    	\end{eqnarray*}
    	with
    	$
    	h := \max\left\{0,\frac{sp-1}{p-1}\right\} < s.
    	$
    	In particular, if~$u$ is a weak supersolution in an open set~$\Omega$, then~$ u \in L^{p-1}_{sp}(\He)$.
    \end{lemma}

    \begin{proof}
Firstly, we write the weak formulation, for nonnegative $\phi \in C_0^\infty(B_{r/2}(\xi_0))$ such that $\phi \equiv 1$ in $B_{r/4}(\xi_0)$, with $0 \leq \phi \leq 1$ and $|\nabla_{\He} \phi| \leq 8/r$. We have
\begin{eqnarray*}  
\int_{\He}f(\xi,u)\phi(\xi) \dxi & \leq & \int_{B_{r}(\xi_0)} \int_{B_r(\xi_0)} \frac{|u(\xi)-u(\eta)|^{p-2} \big(u(\xi)-u(\eta)\big)\big(\phi(\xi)- \phi(\eta)\big)}{\dd(\eta^{-1}\circ \xi)^{Q+sp}}\dxy
\\*[0.5ex] &&  + \int_{\He \setminus B_r(\xi_0)} \int_{B_{r/2}(\xi_0)} \frac{|u(\xi)-u(\eta)|^{p-2} \big(u(\xi)-u(\eta)\big) \phi(\xi)}{\dd(\eta^{-1}\circ \xi)^{Q+sp}} \dxy 
\\*[0.5ex] & =: & I_1 + I_2.
\end{eqnarray*}
The left-hand side can be treated recalling the regularity assumptions on the nonlinearity~$f$
\[
   \int_{\He} f(\xi,u) \phi(\xi) \dxi \geq - c \, r^Q\|f\|_{L^{\infty}(B_r(\xi_0))}.
\]
The first term can be easily estimated using $|\phi(\xi) - \phi(\eta)| \leq 8|\eta^{-1}\circ\xi|_{\He}/r$ to get
\begin{align*}  
I_1 \leq \frac{c}{r^{\min\{sp,1\}}} \left[u\right]_{HW^{h,p-1}(B_{r}(\xi_0))}^{p-1}
\end{align*}
In order to estimate the second term, we can observe that 
\begin{align*}  
|u(\xi)-u(\eta)|^{p-2} \big(u(\xi)-u(\eta)\big) & \leq 2^{p-1}\big( u_+^{p-1}(\xi)  + u_-^{p-1}(\eta) \big) - u_+^{p-1}(\eta),
\end{align*}
and thus 
\begin{eqnarray*}  
I_2 & \leq & c\, r^{-sp} \left\| u \right\|_{L^{p-1}(B_r(\xi_0))}^{p-1} + c\, r^{Q-sp} {\rm Tail}(u_- ; \xi_0,r)^{p-1} \\
&&  - \frac{r^{Q-sp}}c {\rm Tail}(u ; \xi_0,r)^{p-1}.
\end{eqnarray*}
By combining all the displays above, we obtain the desired estimate. The second statement plainly follows by an application of H\"older's Inequality. 
\end{proof}

    \vspace{2mm}
   
	%
	%
	\section{The De Giorgi-Nash-Moser regularity theory for weak solutions}\label{sec_hold}
	In this section we present some recent results about regularity estimates (up to the boundary) for weak solutions to~\eqref{pbm}. In particular, we focus on De Giorgi-Nash-Moser-type results.  
		
	\subsection{Fundamental estimates} 
 We start recalling some fundamental estimates contained in~\cite{MPPP23}. The first is a Caccioppoli-type inequality which takes into account the presence of the nonlocal tail.
	
	\begin{thm}[{\bf Caccioppoli estimates with tail}. See~Theorem~1.3 in~\cite{MPPP23}]\label{teo_caccioppoli}
		\mbox{}\\ Let~$p\in(1,\infty)$,~$s \in (0,1)$ and let~$u \in HW^{s,p}(\He)$ be a weak subsolution to~\eqref{pbm}. Then, for any~$B_r \equiv B_r(\xi_0) \subset \Omega$ and any nonnegative~$\ph \in C^\infty_0(B_r)$, the following estimate holds true		
		\begin{eqnarray*}
			&&\int_{B_r} \int_{B_r}\frac{|\omega_+(\xi)\ph(\xi)-\omega_+(\eta)\ph(\eta)|^p}{d_{\rm o}(\eta^{-1}\circ \xi)^{Q+sp}} \, \,{\rm d}\xi{\rm d}\eta\\*
			&&\qquad\qquad \leq c\int_{B_r}\int_{B_r}\frac{ \omega_+^p(\xi)|\ph(\xi)-\ph(\eta)|^p}{ d_{\rm o}(\eta^{-1}\circ \xi)^{Q+sp}} \, \,{\rm d}\xi{\rm d}\eta\\*
			&&\qquad\qquad \quad + c \int_{B_r}\omega_+(\xi)\ph^p(\xi) \, \,{\rm d}\xi \biggl(\sup_{\eta \in \textup{supp}\, \ph}\int_{\He \smallsetminus B_r}  \frac{\omega_+^{p-1}(\xi)}{d_{\rm o}(\eta^{-1}\circ \xi)^{Q+sp}} \, \,{\rm d}\xi+\|f\|_{L^\infty(B_r)}\biggr),
		\end{eqnarray*}
		where $\omega_+ := (u-k)_+$ with $k\in\er$, and $c\equiv c\,(
		n,p)>0$.
	\end{thm}
    
    \begin{rem} 
    	The same estimate provided by Theorem~{\rm\ref{teo_caccioppoli}} above holds true if one considers weak supersolutions by replacing~$\omega_+$ with~$\omega_-:=(u-k)_-$.
    \end{rem}
	
	As mentioned above, in the nonlocal framework we have to take into account the contributions coming from far. This is done via the second term on the right-hand side of the Caccioppoli inequality; i.~\!e., 
	$$
	\sup_{\eta \in \textup{supp}\, \ph}\int_{\He \smallsetminus B_r} \frac{\omega_+^{p-1}(\xi)}{d_{\rm o}(\eta^{-1}\circ \xi)^{Q+sp}} \, \,{\rm d}\xi \le c\,r^{-sp}\Tail(w_+;\xi_0,r)^{p-1}.
	$$ 
	As for classical De Giorgi-Nash-Moser theory, the following fractional logarithmic lemma is fundamental in order to prove an H\"older continuity estimates.
	
	\begin{lemma}[{\bf Fractional logarithmic Lemma}. See~Lemma~1.4 in~\cite{MPPP23}]\label{lem_log}
		\mbox{} \\ Let~$p\in (1,\infty)$,~$s \in (0,1)$ and let~$u \in HW^{s,p}(\He)$ be a weak solution to~\eqref{pbm} such that~$u \geq 0$ in~$B_{R} \equiv B_R(\xi_0) \subset \Omega$. Then, for any~$B_r \equiv B_r(\xi_0) \subset B_{\frac{R}{2}}(\xi_0)$ and any~$d>0$,
		\begin{eqnarray*}
			&& \int_{B_r}\int_{B_r} \, \frac{\big|\log (u(\xi)+d)-\log(u(\eta)+d)\big|^p}{d_{\rm o}(\eta^{-1}\circ\xi)^{Q+sp}} \,{\rm d}\xi{\rm d}\eta 
			+\int_{B_r}\big(f(\xi,u)\big)_+\big(u(\xi)+d\big)^{1-p}\,{\rm d}\xi\\*     
			&&\qquad\qquad\qquad\qquad\qquad\qquad \quad \leq c r^{Q-sp}\left[1+ d^{1-p}\left(\frac{r}{R}\right)^{sp}\Big(\Tail(u_-;\xi_0,R)^{p-1}+1\Big)\right]\\*
			&&\qquad \qquad\qquad\qquad\qquad\qquad\qquad +\, c\|f\|_{L^\infty(B_{r})}\int_{B_{2r}}(u(\xi)+d)^{1-p} \,{\rm d}\xi,
		\end{eqnarray*}
		where $\Tail(\cdot)$ is defined in~\eqref{tail},~$u_-:=\max\{-u,0\}$, and~$c\equiv c(n,p,s,\Lambda)>0$.
	\end{lemma}

	\subsection{Local boundedness and H\"older continuity}
	As in the classical regularity theory, the Caccioppoli inequality is the starting point in order to build a proper iteration scheme which will lead to establish an $L^\infty-L^p$ bound. 
    \begin{thm}[{\bf Local boundedness}. See Theorem 1.1 in~\cite{MPPP23}]\label{teo_bdd}
    	Let~$p \in (1,\infty)$,~$s\in (0,1)$, let~$u \in HW^{s,p}(\He)$ be a weak subsolution to~\eqref{pbm}, and let~$B_r \equiv B_r(\xi_0)   \subset \Omega$. Then the following estimate holds true, for any~$\delta \in (0,1]$,
    	\begin{equation}\label{eq_bdd}
    		\esssup_{B_{r/2}}u \, \leq\,   \delta \,\Tail(u_+; \xi_0, r/2) + c\,\delta^{-\frac{(p-1)Q}{sp^2}}  \left(\,\dashint_{B_r}u_+^p \, {\rm d}\xi\right )^\frac{1}{p}, 
    	\end{equation}
    	where~$\Tail(\cdot)$ is defined in~\eqref{tail},~$u_+:=\max\left\{u,\, 0\right\}$ is the positive part of the function~$u$, and the positive constant~$c$ depends only on $n,s,p,\|f\|_{L^\infty},$ and $\Lambda$.
    \end{thm}  
    We stress that the presence of the parameter~$\delta$ allows a precise interpolation between the local and nonlocal terms in~\eqref{eq_bdd}. Moreover, to the best of our knowledge, the boundedness result presented in Theorem~\ref{teo_bdd} above is new even in the linear case when~$p=2$.
    
    In the non-fractional setting when~$s=1$  other boundedness estimates (of the horizontal gradient) are available, see e.~\!g.~\cite{Muk21,MZ21} for the case of nonlinear equations in the Heisenberg group~$\He$ modeled on the $p$-subLaplacian; see also \cite{CDG93} for more general nonlinear operators in Carnot-Caratheodory spaces, and~\cite{CM22} for general quasi-linear equations with H\"ormander vector fields having $p$-Laplacian growth conditions.
    
    \vspace{2mm}
    Combining all the results above, we can finally extend the classical results by De Giorgi-Nash-Moser in the nonlocal Heisenberg framework.  
    \begin{thm}[{\bf H\"older continuity}. See Theorem 1.2 in~\cite{MPPP23}]\label{teo_holder}
     Let~$p\in(1,\infty)$, $s \in (0,1)$, and let~$u \in HW^{s,p}(\He)$ be a weak solution to~\eqref{pbm}. Then~$u$ is locally H\"older continuous in~$\Omega$. In particular, there are constants~$\alpha < sp/(p-1)$ and~$c>0$, both depending only on~$n,p,s,\Lambda$ and~$\|f\|_{L^\infty(B_r)}$, such that if~$B_{2r}\equiv B_{2r}(\xi_0) \subset \Omega$ then
    	\begin{equation*}
    		\osc\limits_{B_\varrho} \, u 
    		\, \leq\, c\, \left(\frac{\varrho}{r}\right)^\alpha \left[\Tail(u;\xi_0,r) + \left(\,\dashint_{B_{2r}}\snr{u}^p \,{\rm d}\xi\right)^\frac{1}{p}\right],
    	\end{equation*} 
    	for every~$\varrho\in (0,r)$.
    \end{thm}

    \begin{proof}
    We will present a sketch of the proof by conveniently dividing it in a few steps.
    \mbox{}
    \\ {\it Step 1}. \  
     For any~$j \in \mathds{N}$, let~$0 <\rr < R/2$ for~$R$ such that~$B_R \subset \Omega$,
              $$
              \rr_j := \sigma^{\,j} \frac{\rr}{2}, \quad \sigma \in \left(0,\frac{1}{4}\right], \quad B_j \equiv B_{\rr_j}\,
              $$ 

              $$
              \omega(\rr_0) \sim  \Tl(u;\rr/2)+ \left(\,\dashint_{B_\rr}u_+^p\,{\rm d}\xi\right)^\frac{1}{p}
              $$
               and
              $$
              \omega(\rr_j):= \left(\frac{\rr_j}{\rr_0}\right)^\alpha \omega(\rr_0) \qquad \text{for some}~\alpha < \frac{sp}{p-1}.
              $$ 
              It suffices to prove that  under the notation above
              	\begin{equation*}
              		\label{s4 1}
              		\osc\limits_{B_j} \, u \leq \omega(\rr_j), \qquad \forall j =0,1,2,\dots
              	\end{equation*}
and this can be achieved by induction. Indeed,
the case~when~$j=0$ plainly follows by the supremum estimate in~Theorem~\ref{teo_bdd} by taking~$\delta \equiv 1$, since both the functions $(u)_+$ and $(u)_-$ are weak subsolution. 
Assume now that it holds for all~$i \in \big\{1,\dots,j\big\}$ for some~$j \geq 0$ and below we will prove that it holds also for~$j+1$.

   \mbox{}
    \\ {\it Step 2}. \  
        With no loss of generality, one can assume that
              \[ 
              		\frac{\big|{2B_{j+1} \cap \big\{u \geq \essinf_{B_j}u+\omega(\rr_j)/2\big\}}\big|}{\big|{2B_{j+1}}\big|} \,\geq\, \frac{1}{2},
              \]
              	so that, denoting~$u_j := u -\essinf_{B_j}u \geq 0$ in~$B_j$,  we have 
              	\begin{equation*}       	
              		\frac{\big|2B_{j+1}\cap \{u_j \geq \omega(\rr_j)/2\}\big|}{\big|2B_{j+1}\big|}\geq \frac{1}{2} \qquad 
              		\sup_{B_i}|u_j| \leq 2 \omega(\rr_i), \quad \forall i \in \{1,\dots, j\}.
              	\end{equation*}
	      	
              	Consider now the function~$v$ defined as follows,
              \[
              		v := \min \left\{\left(\log \left(\frac{\omega(\rr_j)/2+\sigma^{\frac{sp}{p-1}}}{u_j+\sigma^{\frac{sp}{p-1}}}\right)\right)_+,\, 	\kappa \right\}, \qquad \kappa>0
               \]
               Thanks to the fractional Log-Lemma~\ref{lem_log}, we have
            \[
              	   \dashint_{2B_{j+1}}|v-(v)_{2B_{j+1}}|^p\,{\rm d}\xi              
	\, \leq \,    1.
              	\]
   \mbox{}
    \\ {\it Step 3}. \  Let
   	\begin{align*}
	    \kappa & = \frac{1}{|2B_{j+1} \cap \{u_j \geq \omega(\rr_j)/2\}|}\int_{2B_{j+1} \cap \{u_j \geq \omega(\rr_j)/2\}} \kappa \,{\rm d}\xi\\*[0.5ex]
              		& = \frac{1}{|2B_{j+1} \cap \{u_j \geq \omega(\rr_j)/2\}|}\int_{2B_{j+1} \cap \{v=0\}}\kappa \,{\rm d}\xi \\*[0.5ex]
              		& \leq \frac{2}{|2B_{j+1}|} \int_{2B_{j+1}}(\kappa-v)\,{\rm d}\xi= 2[\kappa-(v)_{2B_{j+1}}]
              	\end{align*}
              	Integrating on~$2B_{j+1} \cap \{v=\kappa\}$, we get
              	\begin{align*}
              		\frac{|2B_{j+1}\cap \{v=\kappa\}|}{|2B_{j+1}|} \kappa & \leq \frac{2}{|2B_{j+1}|}\int_{2B_{j+1}\cap \{v = \kappa\}}[\kappa-(v)_{2B_{j+1}}]\,{\rm d}\xi\\*[0.5ex]
              		& \leq \frac{2}{|2B_{j+1}|}\int_{2B_{j+1}}|v-(v)_{2B_{j+1}}|\,{\rm d}\xi \lesssim 1.
              	\end{align*}
              	Taking~$\kappa \ \approx - \log \big( \sigma^\frac{sp}{p-1}\big)$ 
              	\begin{equation}\label{step0}
              	\frac{|2B_{j+1} \cap \{u_j \leq 2 \sigma^\frac{sp}{p-1}\omega(\rr_j)\}|}{|2B_{j+1}|} 
              		\,\lesssim\, \frac{1}{\log \left(\frac{1}{\sigma}\right)}
              	\end{equation}
\mbox{}
\ {\it Step 4}. A second iterative scheme is now needed. 
              	For any~$i \in \mathds{N}$, let
              	$$
              	\rr^{(i)} := \rr_{j+1}+2^{-i}\rr_{j+1} \quad B^{(i)}\equiv B_{\rr^{(i)}} \quad \p_i \in C^\infty_0~\text{cut-off}  \quad \kappa_i := (1+2^{-i}) \sigma^\frac{sp}{p-1} \omega(r_j) \,.
              	$$ 
 We denote by              	$$
                A_i := \frac{|B^{(i)} \cap \{w_i >0\}|}{|B^{(i)}|} \qquad  \text{for}~w_i := (\kappa_i-u_j)_+.
              	$$
 We have the following estimate,
                	\begin{eqnarray*}
              	&&A_{i+1}^\frac{p}{p^*}(\kappa_i-\kappa_{i+1})^p \\
              	&&\qquad \leq |B^{(i+1)}|^{-\frac{p}{p^*}}\norma{w_i\p_i}_{L^{p^*}(B^{(i)})}^p\\*
              	&& \qquad \leq r^{sp-Q}_{j+1}[w_i\p_i]_{W^{s,p}(B^{(i)})}\\
              	&&\qquad \leq \int_{B^{(i)}}\int_{B^{(i)}}  \frac{w_i^p(\xi)|\p_i(\xi)-\p_i(\eta)|^p}{|\eta^{-1} \circ \xi|_{\He}^{Q+sp}} \, \,{\rm d}\xi \,{\rm d}\eta\\*
              	&&\qquad\qquad  +  \norma{w_i\p_i^p}_{L^1(B^{(i)})} \biggl(\esssup_{\eta \in \textup{supp}(\p_i)}\int_{\He\smallsetminus B^{(i)}} \frac{w_i^{p-1}(\xi) \, \,{\rm d}\xi}{|\eta^{-1} \circ \xi|_{\He}^{Q+sp}} +\norma{f}_{L^\infty}\biggr)\notag\,,
              	\end{eqnarray*}
	where we also used the Folland-Stein-Sobolev embedding and the Caccioppoli inequality in Theorem~\ref{teo_caccioppoli}. It follows,
	$$
              	A_{i+1} \le \text{c}_{\rm o} \text{b}^i A_i^\frac{Q}{Q-sp}, \qquad \text{c}_{\rm o},\text{b} \equiv \text{c}_{\rm o},\text{b}({\tt data}) >1.
$$ 
	Now, by recalling a classical lemma in the Euclidean setting, it remains only to prove that for a proper~$\nu^* \equiv\nu^*({\tt data})$ the following estimate holds,
              	\begin{equation*}
              		\label{s4 18}
              		A_0 := \frac{|2B_{j+1} \cap \{u_j \leq 2\sigma^{\frac{sp}{p-1}} \omega(\rr_j)\}|}{|2B_{j+1}|}
              		\, \leq\, \nu^*.
              	\end{equation*}
                For this, we use the estimate in~\eqref{step0} and a proper choice of~$\sigma \equiv \sigma(\nu^*)$ to get
              	\[
              	\lim_{i \to \infty}A_i = 0 \,.
              	\]
               We have hence shown that 
               \begin{eqnarray*}
               \osc_{B_{j+1}} \, u 
              \,= \, (1-d)\sigma^{-\alpha}\omega(\rr_{j+1}).
               \end{eqnarray*}
               
               Finally, taking~$\alpha \in \left(0, \frac{sp}{p-1}\right)$ small enough so that
               \[
               \sigma^\alpha \geq 1-d=1-\sigma^\frac{sp}{p-1}\,
               \]
             will               lead to the desired estimate,
              \[
              \osc_{B_{j+1}} \, u \leq \omega(\rr_{j+1}).
               \]          
        \end{proof}

    In the linear case, when $p=2$,  for what concerns classical H\"older regularity results for linear integro-differential operators in a very wide class of metric measure spaces, we refer to the important paper by Chen and Kumagai~\cite{CK08}. Moreover, we recall the result proven by Ferrari and Franchi when~$p=2$ for fractional subLaplacians in Carnot group via a Dirichlet-to-Neumann-type extension, which, as mentioned in the introduction, is not available in the more general framework we are dealing with. 
        
    For some further $C^{1,\alpha}$-regularity estimates in the local case when~$s=1$ we refer the interested reader to~\cite{Cap97,CG03,MM07,MZ21,MS21,CM22} and the references therein.

    \subsection{Nonlocal Harnack inequalities} \label{sec_h}
        Now, we are ready to focus on Harnack-type inequalities for  weak solutions to~\eqref{pbm}.  Forthcoming Theorems~\ref{thm_harnack} and~\ref{thm_weak} do generalize to the Heisenberg (and to the non-homogeneous setting) the results obtained by Di Castro, Kuusi and one of the authors in~\cite{DKP14} for nonlinear and nonlocal integro-differential operators with  $(s,p)$-kernels. The proofs in~\cite{DKP14} rely on the supremum estimate obtained in Theorem~\ref{teo_bdd} together with a suitable expansion of positivity which takes into account the presence of the nonlocal tail; a finite Moser iteration argument will give the desired results.
     Further efforts are needed in order to deal with the limiting case when~$sp=Q$ of the fractional Sobolev embedding; see Theorem~\ref{sobolev} above. 

    \begin{thm}[{\bf Nonlocal Harnack inequality}. See Theorem 1.1  in~\cite{PP22}]\label{thm_harnack}
    \mbox{}\\    	For any~$s \in (0,1)$ and any~$p \in (1,\infty)$, let~$u \in HW^{s,p}(\He)$ be a weak solution to~\eqref{pbm} such that~$u \geq 0$ in~$B_R \equiv B_R(\xi_0) \subset \Omega$. Then, for any~$B_r$ such that~$B_{6r}\subset B_R$, it holds
    	\begin{equation}\label{eq_harnack}
    		\esssup_{B_{r}}u 
    		\, \leq \,
    		c\, \essinf_{B_r}u + c \left(\frac{r}{R}\right)^\frac{sp}{p-1}\Tl(u_-;\xi_0,R) +c\,r^\frac{sp}{p-1}\|f\|^\frac{1}{p-1}_{L^\infty(B_R)}\,,\\[1.3ex]
    	\end{equation}
    	where~$\Tail(\cdot)$ is defined in~\eqref{tail},~$u_-:=\max\{-u,0\}$ is the negative part of the function~$u$, and~$c \equiv c(n,s,p,\Lambda)>0$.
    \end{thm}
    
    Moreover, in the case when~$u$ is merely a weak supersolution to~\eqref{pbm}, in clear accordance with the Euclidean framework, the following weak Harnack inequality can be proved.
    \begin{thm}[{\bf Nonlocal weak Harnack inequality}. See Theorem 1.2 in~\cite{PP22}]\label{thm_weak}
    \mbox{}\\	For any~$s \in (0,1)$ and any~$p \in (1,\infty)$, let~$u \in HW^{s,p}(\He)$ be a weak supersolution to~\eqref{pbm} such that~$u \geq 0$ in~$B_R \equiv B_R(\xi_0) \subseteq \Omega$.
    	Then, for any~$B_r$ such that~$B_{6r}\subset B_R$, it holds
    	\begin{equation}\label{eq_weakharnack}
    		\left( \ \dashint_{B_{r}}  u^\mathfrak{t}\,{\rm d}\xi\,\right)^\frac{1}{\mathfrak{t}} \ \leq \ c\,\essinf_{B_{\frac{3}{2}r}} u 
    		\, +c \left(\frac{r}{R}\right)^\frac{sp}{p-1} \Tail(u_-;\xi_0,R) \, + c\,\mathcal{F},
    	\end{equation}
    	where
    	$$
    	\mathcal{F} := 
    	\begin{cases}
    		r^\frac{Qsp}{\mathfrak{t}(Q-sp)}\|f\|_{L^\infty(B_R)}^\frac{Q}{\mathfrak{t}(Q-sp)} \qquad \text{for} \ \mathfrak{t}<\frac{Q(p-1)}{Q-sp} &\quad  \text{if} \ sp<Q,\\[0.8ex]
    		r^\frac{Q(s-\varepsilon)}{\mathfrak{t}\varepsilon}\|f\|_{L^\infty(B_R)}^\frac{s}{\mathfrak{t}\varepsilon}  \qquad \text{for any} \ s-{Q}/{p}< \varepsilon<s \ \text{and}\  \mathfrak{t}<\frac{(p-1)s}{\varepsilon} &\quad  \text{if} \ sp\geq Q,
    	\end{cases}
    	$$
    	$\Tail(\cdot)$ is defined in~\eqref{tail}, and~$u_-:=\max\{-u,0\}$ is the negative part of the function~$u$ and~$c \equiv c(n,s,p,\Lambda)>0$.
    \end{thm}
  
  \begin{rem}
The nonlocal Harnack inequalities presented above have been stated in the original paper~\cite{PP22} under a not sharp assumption; that is, for {$1<p< 2n/(1-s)$}. One can easily remove such a restriction by bypassing the estimate {\rm (3.1)} there given by~\cite[Lemma~3.2]{AM18}  with the estimate in {\rm Proposition~5.4.4} by Bonfiglioli, Lanconelli and Uguzzoni~\cite{BLU07}. Such a simple (sharp) computation can be found at {\,\rm Page~24} in~\cite{MPPP23}.
  \end{rem}
    In the linear case when~$p=2$ a first version of the Harnack inequality has been proved by Ferrari and Franchi in~\cite{FF15}, via a Dirichlet-to-Neumann map, and for  weak solutions nonnegative in the whole space. In this regard, it is worth noticing that by assuming~$u \ge 0$ in~$\He$ the~$\Tail (\cdot)$-type contribution in both~\eqref{eq_harnack} and~\eqref{eq_weakharnack} disappears.

    \subsubsection{The conformally invariant fractional subLaplacian}
    In the Euclidean framework, the validity of the Harnack inequality without requiring extra positivity assumptions on the solutions has been an open problem in the fractional framework for many years, until Ka$\ss$mann in~\cite{Kas07,Kas11} showed with a simple counterexample that the classical Harnack inequality is not true even for the classical fractional Laplacian~$(-\Delta)^s$ if one cannot add an extra tail-type contribution on the right-hand side of the desired inequality. 
    
    \vspace{2mm}
    We show now that in such linear case when~$p=2$ such a tail term disappers in both the inequalities~\eqref{eq_harnack} and~\eqref{eq_weakharnack} as~$s\to 1$, so that the classic Harnack inequalities are recovered.

    Precisely, we consider the following Dirichlet problem
     \begin{eqnarray}\label{fractional_pbm}
    	\begin{cases}
    		(-\Delta_{\He})^{s} u = 0  & \textrm{in}~\Omega,\\[0.4ex]
    		u = g  & \textrm{in}~\He \smallsetminus\Omega,
    	\end{cases}
    \end{eqnarray}
    where $g \in H^s(\He)\equiv HW^{s,2}(\He)$. We firstly recall the  definition of the conformally invariant fractional subLaplacian operator
   $$
    	(-\Delta_{\He})^{s} u(\xi) = C(n,s)\ P.~\!V.\int_{\He} \frac{u(\xi)-u(\eta)}{|\eta^{-1} \circ \xi |_{\He}^{Q+2s}}\, {\rm d}\eta, \qquad \forall \xi \in \He,
    $$
    where~$C(n,s)$ is given by
     \begin{equation}\label{C(n,s)}
    	C(n,s) = \frac{c_1(n,s)\omega_{2n}}{n c_2(n,s)},
    \end{equation} 
    with
    \begin{equation}\label{c(n,s)}
    	c_1(n,s) = \left(\, \int_{\er^{2n+1}} \frac{1-\cos(x_1)}{\|\eta\|^{1+2(n+s)}}\,{\rm d}\eta\right)^{-1} \quad \text{and} \quad
    	c_2(n,s) = \int_{\partial B_1}\frac{x_1^2}{|\eta|_{\h^n}^{Q+2s}}\,{\rm d}\mathcal{H}_{Q-2}(\eta)\,,
    \end{equation}
  In the display above,~$\|\cdot\|$ is the standard Euclidean norm on~$\er^{2n+1}$, and~$\mathcal{H}_{Q-2}$ is the surface measure on~$\partial B_1$.
    
    \vspace{2mm}
    With such a notation in mind, we can prove that the conformally invariant subLaplacian tends to the classical subLaplacian~$-\Delta_{\He}$ on suitable regular functions.
    
       \begin{prop}[{\bf Asymptotic of the conformally fractional subLaplacian}. See Proposition~1.3 in~\cite{PP22}]\label{limit_of_fraclap} \mbox{}
    	For any $u \in C^\infty_0(\He)$ the following statement holds true
    	$$
    		\lim_{s \rightarrow 1^-}(-\Delta_{\He})^{s} u = -\Delta_{\He} u\,.
    	$$
    \end{prop}
  
   \begin{proof}
    For the sake of readability, we denote the points~$\xi$ in $\He$ as follows,
    $$
    \xi := (x_1,\dots,x_{2n},t)\,,
    $$
and we rely on the weighted second order integral definition of the fractional sublaplacian; i.~\!e.,
    $$
    (-\Delta_{\h^n})^s u (\xi) \, =\, -\frac{1}{2}C(n,s) \int_{\h^n}\frac{u(\xi \circ \eta)+ u(\xi \circ \eta^{-1} )-2u(\xi)}{|\eta|_{\h^n}^{Q+2s}}\,{\rm d}\eta, \quad \forall \xi \in \h^n\,;
    $$
    see, e.~\!g.,~\cite[Proposition~3.2]{DPV12}.
    We also recall that the symmetrized horizontal Hessian matrix~$D^{2,*}_{\h^n}$ is given by
    \[
    	D^{2,*}_{\h^n} u(\xi) := \left(\frac{1}{2}\big(X_iX_ju(\xi)+X_jX_iu(\xi)\big)\right)_{i,j =1,\dots,2n}\,.
    \]

   Firstly, we have no contribution outside the unit ball in the limit as $s$ goes to $1^-$. Indeed,
    \[
    	\lim_{s \rightarrow 1^-}-\frac{C(n,s)}{2}\int_{\h^n \smallsetminus B_1(0)}\frac{u(\xi \circ \eta)+ u(\xi \circ \eta^{-1} )-2u(\xi)}{|\eta|_{\h^n}^{Q+2s}} \, {\rm d}\eta\,=\,0.
    \]
    It thus remains to estimate the integral contribution in the unit ball. In view of~\cite[Corollary~20.3.5]{BLU07}, for any $\eta =(x,t)$, one get
    \begin{equation}\label{ss3_e4}
    	u(\xi \circ \eta^{-1})= P_2(u,\xi)(\xi \circ \eta^{-1}) + {\text{o}}(|\eta|_{\h^n}^3) \quad \text{as} \ |\eta|_{\h^n}\to 0\,,
    \end{equation}
    where $P_2(u,\xi)$ is the Taylor polynomial of $\h^n$-degree~2 associated to~$u$ and centered at~$\xi$; see~\cite{BLU07}.

    By the very definition of Taylor polynomial, inequality~\eqref{ss3_e4} yields
    \[
    	u(\xi \circ \eta^{-1}) =  u(\xi)- \big(\nabla_{\h^n}u(\xi),\,\partial_t u(\xi)\big)\cdot \eta + \frac{1}{2}\langle x,D^{2,*}_{\h^n}u(\xi) \cdot x\rangle + {\text{o}}(|\eta|_{\h^n}^3)\quad \text{as} \ |\eta|_{\h^n}\to 0\,,
    \]
    which, by using again the result in~\cite{BLU07}, can be used to control the contribution inside the ball~$B_1$ as follows,
    \begin{eqnarray*}
    	&& \left|\quad \int_{B_1}\frac{u(\xi \circ \eta)+ u(\xi \circ \eta^{-1} )-2u(\xi)-\langle x ,D^{2,*}_{\h^n}u(\xi) \cdot x \rangle}{|\eta|_{\h^n}^{Q+2s}} \, {\rm d}\eta \quad \right|\notag\\*[0.5ex]
    	&& \qquad\qquad\qquad\qquad\qquad\quad \leq \int_{B_1}\frac{|u(\xi \circ \eta) -P_2(u(\xi \circ \cdot),0)(\eta)|+{\text{o}}(|\eta|_{\h^n}^3)}{|\eta|_{\h^n}^{Q+2s}} \, {\rm d}\eta\notag\\*[0.5ex] 
    	&& \qquad\qquad\qquad\qquad\qquad\quad \leq \int_{B_1}\frac{1}{|\eta|_{\h^n}^{Q-1}}\  =:\  c(n).
    \end{eqnarray*}
    The previous  estimate yields
    \begin{eqnarray}\label{ss3_e5}
    	&& \lim_{s \rightarrow 1^-}-\frac{C(n,s)}{2}\int_{B_1}\frac{u(\xi \circ \eta)+ u(\xi \circ \eta^{-1} )-2u(\xi)}{|\eta|_{\h^n}^{Q+2s}} \, {\rm d}\eta \notag\\*
    	&& \qquad  = \lim_{s \rightarrow 1^-} -\frac{C(n,s)}{2}\int_{B_1}\frac{\langle x,D^{2,*}_{\h^n}u(\xi) \cdot x \rangle }{|\eta|_{\h^n}^{Q+2s}}\,{\rm d}\eta.
    \end{eqnarray}
    
    Now, notice that 
    \begin{equation}\label{ss3 _e6}
    	\int_{B_1}\left(\frac{1}{2}(X_iX_ju(\xi)+X_jX_iu(\xi))\right) x_i \,  \cdot \, x_j {\rm d}\eta =0, \qquad \text{for}\ i \neq j.
    \end{equation}
    Moreover, for any fixed index~$i$, making using of the polar coordinates, namely~\cite[Proposition~1.15]{FS82}, we get that there exists a unique Borel measure~$\sigma$ on $B_1$ such that, up to permutations,
    \begin{eqnarray*} 
    	\int_{B_1}\frac{X_i^2 u(\xi) x_i^2}{|\eta|_{\h^n}^{Q+2s}}\, {\rm d}\eta & = & X_i^2 u(\xi) \int_{0}^{1}\int_{\partial B_1}\frac{x_1^2r^{Q+1}}{|{\Phi}_{r}(\eta)|_{\h^n}^{Q+2s}}\, {\rm d}\sigma(\eta)\,{\rm d}r\notag\\*[0.8ex] 
    	& =& \frac{c_2(n,s)}{2(1-s)}X_i^2u(\xi).
    \end{eqnarray*}
    where~$c_2(n,s)$ is defined in~\eqref{c(n,s)}.
    \vspace{1mm}
    
    Hence, recalling the result in~\cite[Corollary~4.2]{DPV12} which shows that
    \begin{equation}\label{limit_c(n,s)}
    	\lim_{s \rightarrow 1^-}\frac{c_1(n,s)}{s(1-s)}= \frac{4n}{\omega_{2n}},
    \end{equation}
    where~$\omega_{2n}$ denotes the~$(2n)$-dimensional Lebesgue measure of the unit sphere~$\mathds{S}^{2n}$, we finally obtain that
    \begin{eqnarray*}
    	\lim_{s \rightarrow 1^-}(-\Delta_{\h^n})^s u(\xi)  & = &\lim_{s \rightarrow 1^-} -\frac{C(n,s)}{2}\int_{B_1} \frac{\langle x,D^{2,*}_{\h^n}u(\xi)  \cdot x\rangle }{|\eta|_{\h^n}^{Q+2s}}\,{\rm d}\eta\\*
    	& =& \lim_{s \rightarrow 1^-} -\frac{c_1(n,s)\omega_{2n}}{4n(1-s)} \sum_{i=1}^{2n}X_i^2 u(\xi) \ \equiv \  -\Delta_{\h^n}u(\xi)\,,
    \end{eqnarray*}
    as desired.
\end{proof}

\begin{rem}   
   A proof of the result above for fractional subLaplacians on Carnot groups can be also achieved via  heat kernel characterization;
see the relevant paper~\cite{FMPPS18}.
   \end{rem}
   
    \vspace{2mm}
    
    With Proposition~\ref{limit_of_fraclap} in mind, we are now able to revisit the proofs of Theorems~\ref{thm_harnack}-\ref{thm_weak} by precisely  tracking the dependence of the differentiability exponent~$s$ in all the estimates, so that we are eventually able to obtain the results below in clear accordance with the analogous ones in the Euclidean framework~\cite{Kas11}, by proving that the nonlocal tail term will vanish when $s$~goes to~$1$, in turn recovering the classical Harnack inequality formulation.
    \begin{thm}[{\bf Robustness of the nonlocal Harnack inequality}. See Theorems~1.4 and 1.5 in~\cite{PP22}] 
        	For any~$s \in (0,1)$, let~$u \in H^s(\He)$ be a weak solution to~\eqref{fractional_pbm} such that~$u \geq 0$ in $ B_R(\xi_0) \subset \Omega$. Then, the following estimate holds true for any~$B_{r}$ such that~$B_{6r} \subset B_R$,
    	\begin{equation*}
    		\esssup_{B_{r}}u \leq c\, \essinf_{B_{r}}u + c\, (1-s)\left(\frac{r}{R}\right)^{2s}\Tail(u_-;\xi_0,R).
    	\end{equation*}
        If~ $u \in H^s(\He)$ is a weak supersolution to~\eqref{fractional_pbm}, such that~$u \geq 0$ in~$B_R(\xi_0) \subseteq \Omega$, then,  for any~$B_r$ such that~$B_{6r} \subset B_R$, and any~$\mathfrak{t}<Q/(Q-2s)$,
    	\begin{equation*}
    		\left( \, \dashint_{B_{r}}  u^\mathfrak{t}{\rm \, d}\xi\right)^\frac{1}{\mathfrak{t}} \leq  c\,\essinf_{B_{\frac{3r}{2}}} u +c \,(1-s)\left(\frac{r}{R}\right)^{2s} \Tail(u_-;\xi_0,R),
    	\end{equation*}
    	where~$\Tail(\cdot)$ is defined in~\eqref{tail} by taking~$p=2$ there,~$u_-:=\max\{-u,0\}$ is the negative part of the function~$u$, and~$c\equiv c\,(n,s)$.
    \end{thm}

\vspace{2mm}
    \section{The subelliptic obstacle problem}\label{sec_obs}
    Armed with all the estimates presented in Sections~\ref{sec_hold}, we are in the position to investigate {\it the  obstacle problem} associated to the operators~$\Lc$ in order to extend (up to the boundary) both the boundedness and the H\"older continuity stated in Theorems~\ref{teo_bdd} and~\ref{teo_holder}.
    For other results regarding the obstacle problem in subRiemannian geometry we refer to~\cite{DGP07,DGS03} and the references given there.
    
    \vspace{2mm}
    Let~$\Omega \Subset \Omega'$ be bounded open sets in $ \He$; let~$h : \He \rightarrow [-\infty,\infty)$ be an extended real-valued function, called the {\it obstacle}; and let~$g \in HW^{s,p}(\Omega')\cap L^{p-1}_{sp}(\He)$ be the nonlocal boundary datum. Define the following class of fractional functions, 
    \begin{equation*} 
    	\mathcal{K}_{g,h}(\Omega,\Omega') := \big\{ u \in HW^{s,p}(\Omega'): u \geq h \mbox{ a.~\!e. in } \Omega, \, u=g \mbox{ a.~\!e. in } \He \smallsetminus\Omega\big\}.
     \end{equation*}
   Note that in the case when~$h \equiv -\infty$, we have that the  class defined above, $\mathcal{K}_{g,h}(\Omega,\Omega')$, does coincide with
   $$
   \mathcal{K}_{g,-\infty}(\Omega,\Omega')= \bigl\{ u \in HW^{s,p}(\Omega'): u=g ~\mbox{a.~\!e. in}~\He \smallsetminus\Omega\big\},
   $$
   which is the class  
    in which we seek the solution to~\eqref{pbm} with zero datum in the right-hand side of the equation; see Definition~\ref{solution to inhomo pbm} above, and forthcoming Section~\ref{sec_laplace}.
  
   On the class~$\mathcal{K}_{g,h}(\Omega,\Omega')$ we now define the functional
   $$
   \mathcal{A}: \mathcal{K}_{g,h}(\Omega,\Omega') \rightarrow [HW^{s,p}(\Omega')]',
   $$
    given by
    \begin{eqnarray*}
    	&& \langle\mathcal{A}(u),v\rangle\\*
    	&&\quad :=  \int_{\He}\int_{\He}\frac{|u(\xi)-u(\eta)|^{p-2}(u(\xi)-u(\eta)) (v(\xi)-v(\eta))}{d_{\rm o}(\eta^{-1}\circ \xi)^{Q+sp}} \,{\rm d}\xi{\rm d}\eta  - \int_{\Omega} f(\xi,u)v(\xi) \,{\rm d}\xi,
    \end{eqnarray*}
     where the function~$f:\He \times \er \rightarrow \er$ satisfies the following assumptions:
    \begin{enumerate} 
    	 \item[($F_1$)]   $   f (\cdot,\tau)$ is measurable for all~$\tau \in \er$ and~$f(\xi,\cdot)$ is continuous for a.~\!e.~$\xi \in \He$; 
    	 \item[($F_2$)]   $	f \equiv f(\cdot,\tau) \in L^\infty_{\rm loc}(\He)$, for any~$\tau\in \er$, uniformly in~$\Omega$;
    	 \item[($F_3$)]   $\left(f(\xi,\tau_1)-f(\xi,\tau_2)\right)(\tau_1-\tau_2) \leq 0$, for any~$ \tau_1,\tau_2 \in \er$ and for a.~\!e.~$\xi \in \He$. 
    \end{enumerate}
    We are now ready to recall the definition of solutions to the obstacle problem.
    
    \begin{defn}[{\bf Solution to the obstacle problem}. Definition~2.8 in~\cite{Pic22}]
     We say that~$u \in \mathcal{K}_{g,h}(\Omega,\Omega')$ is a {\it solution to the obstacle problem in~$\mathcal{K}_{g,h}(\Omega,\Omega')$} if
    $$
    \langle\mathcal{A}(u),v-u \rangle \geq 0, \qquad \forall v \in \mathcal{K}_{g,h}(\Omega,\Omega').
    $$
    \end{defn}
   In the theorem below, we immediately state both existence and uniqueness of the solution to the obstacle problem.  
    \begin{thm}[{\bf Existence and uniqueness}. Theorem~1.1 in~\cite{Pic22}] 
    	Under conditions~$(F_1)$, $(F_2)$ and~$(F_3)$, if the class~$\mathcal{K}_{g,h}(\Omega,\Omega')$ is not empty, then there exists a unique solution to the obstacle problem in~$\mathcal{K}_{g,h}(\Omega,\Omega')$.
    \end{thm}
    Also, in clear accordance with the results in the local framework, one can prove that the solutions to the obstacle problem are indeed weak solutions to~\eqref{pbm} away from the contact set. 
    We have the following
    \begin{corol}[{\bf Solutions to the obstacle problem}. See Corollary 1.2 in~\cite{Pic22}]\label{properties_obst_sol}
    	Let~$u$ be the solution to the obstacle problem in~$\mathcal{K}_{g,h}(\Omega,\Omega')$. Then~$u$ is a weak supersolution to~\eqref{pbm} in~$\Omega$.
    	Moreover, if~$B_r(\xi_0)\subset \Omega$ is such that
    	$
      \essinf_{B_r(\xi_0)}(u-h) >0,
    	$
    	then~$u$ is a weak solution to~\eqref{pbm} in~$B_r(\xi_0)$. In particular, if~$u$ is lower semicontinuous and~$h$ is upper semicontinuous in~$\Omega$, then~$u$ is a weak solution to~\eqref{pbm} in~$\{u >h\} \cap \Omega$.
    \end{corol}
Amongst other results, one can prove that the solutions to the obstacle problem inherit regularity properties, such as boundedness, continuity and H\"older continuity
(up to the boundary), from the obstacle.
       \begin{thm}[{\bf Boundedness for the obstacle problem}. Theorem~3.6 in~\cite{Pic22}]\label{bound_bdd}
       Let~$ s \in (0,1)$ and~$p \in (1,\infty)$. 
	    	Suppose that, under assumptions~$(F_1)$,~$(F_2)$ and~$(F_3)$,~$u \in \mathcal{K}_{g,h}(\Omega,\Omega')$ solves the obstacle problem in~$\mathcal{K}_{g,h}(\Omega,\Omega')$. Let~$\xi_0 \in \partial \Omega$ and~$r \in (0,{\rm dist}(\xi_0,\partial \Omega)$, and assume that, for~$ B_r(\xi_0)$, 
    	$$
    	\esssup_{B_r(\xi_0)}g + \esssup_{B_r(\xi_0) \cap \Omega} h <\infty \quad \mbox{and} \quad \essinf_{B_r(\xi_0)} g > -\infty.
    	$$
    	Then,~$u$ is essentially bounded close to~$\xi_0$.
    \end{thm}
	In order to prove the H\"older continuity up to the boundary, a natural measure density condition on~$\Omega$ has to be assumed; that is, there exists~$\delta_\Omega \in (0,1)$ and~$r_0>0$ such that, for every~$\xi_0 \in \partial \Omega$,
	\begin{equation}\label{condizione misura}
		\inf_{r \in (0,r_0)} \frac{|(\He \smallsetminus\Omega)\cap B_r(\xi_0)|}{|B_r(\xi_0)|} \geq \delta_\Omega.
	\end{equation}
	The requirement above is in the same spirit of the classical nonlinear Potential Theory and --~as expected in view of the nonlocality of the involved equations~-- it is translated into an information given  on the complement of the set~$\Omega$. 
   We finally have 
   \begin{thm}[{\bf Regularity for the obstacle problem}. Theorem~1.3 in~\cite{Pic22}]\label{bound_hold}
   	Let~$s \in (0,1)$,~$p \in (1,\infty)$ and, under assumptions~$(F_1)$,~$(F_2)$ and~$(F_3)$, let~$u$ solves the obstacle problem in~$\mathcal{K}_{g,h}(\Omega,\Omega')$ with~$g \in \mathcal{K}_{g,h}(\Omega,\Omega')$ and $\Omega$ satisfying condition~\eqref{condizione misura}. If~$g$ is locally H\"older continuous  (resp. continuous) in~$\Omega'$ and~$h$ is locally H\"older continuous  (resp. continuous) in~$\Omega$ or~$h \equiv -\infty$, then~$u$ is locally H\"older continuous (resp. continuous) in~$\Omega'$.
   \end{thm}

   \begin{proof}
 With no loss of generality we may suppose $\xi_0=0$ and $g(0)=0$. Moreover, we may choose $R_0$ such that $\osc_{B_0}g \leq \osc_{B_0}u$ for $B_0 \equiv B_{R_0}(0)$ since otherwise we have nothing to prove, and  define 
\[
	\omega_0 := 8 \left(\osc_{B_0}u + {\rm Tail}(u;0,R_0) +\frac{ \|f\|^\frac{1}{p-1}_{L^\infty(B_0)}}{\vartheta}\right)\,,
\]
for some~$\vartheta \in (0,1)$.

The proof relies on a logarithmic estimate with tail~\cite[Lemma~3.7]{Pic22}, which follows the same argument of~Lemma~\ref{lem_log} and it is then obtained by a proper choice of the test functions and by a careful estimates of the various energy contributions depending on the range of the integrability exponent~$p$. Such a logarithmic lemma can be subsequently extended to truncations of the solution to the obstacle problem, as follows: let
$B_R \Subset \Omega'$, let $B_r \subset B_{R/2}$ be concentric balls and let
\[
\infty > k_+ \geq \max\bigg\{\esssup_{B_R} g, \esssup_{B_R \cap \Omega} h\bigg\} \quad \text{and} \quad -\infty < k_- \leq \essinf_{B_R} g;
\]
 then the functions $w_\pm:=\esssup_{B_R}(u-k_\pm)_\pm-(u-k_\pm)_\pm+\eps$
satisfy the following estimate
\begin{eqnarray*} 
&&	\int_{B_r}\int_{B_r} \left|\log \frac{w_\pm(\xi)}{w_\pm(\eta)}\right|^p \dd(\eta^{-1} \circ \xi)^{-Q-sp} \dxieta \notag\\*[0.5ex]
&&\qquad \leq c \, r^{Q-sp} \left(1+ r^{sp}\eps^{1-p}\|f\|_{L^\infty(B_R)}+\eps^{1-p} \left(\frac{r}{R}\right)^{sp} {\rm Tail}((w_\pm)_-;\xi_0,R)^{p-1}\right),
\end{eqnarray*}
for every $\eps>0$. Then, we combine the estimate in the display above with a fractional Poincar\'e-type inequality~\cite[Lemma~3.9]{Pic22}, whose proof can be done extending the one presented in~\cite[Lemma~7]{KKP16},  together with some estimates for the tail term thanks to an application of the Chebyshev inequality and in view of the boundedness result up to the boundary in Theorem~\ref{bound_bdd}. 

Eventually we arrive to prove that there exist~$\tau_0$, $\sigma$ and $\vartheta$ depending only on $n$,~$p$,~$s$ and~$\delta_\Omega$ such that if
\[
	\osc_{B_r(0)} u + \sigma \, {\rm Tail}(u;0,r) \leq \omega, \quad \osc_{B_r(0)}g \leq \frac{\omega}{8} \quad \mbox{and} \quad \|f\|_{L^\infty (B_r(0))} \leq \left(\vartheta \omega \right)^{p-1},
\]
hold for a ball~$B_r(0)$ and for~$\omega>0$, then
\[
	\osc_{B_{\tau r}(0)} u + \sigma \, {\rm Tail}(u;0,\tau r) \leq (1-\vartheta)\omega,
\]
holds for every~$\tau \in (0,\tau_0]$.  Finally, as we can take $\tau \leq \tau_0$ such that
$$
\osc_{\tau^{j}B_0}g \leq (1-\theta)^{j}\frac{\omega_0}{8} \qquad \text{for every } j=0,1,\dots,
$$
an iterative argument will give that $u$ belongs to $C^{0,\alpha}(B_0)$. 
\end{proof}

     Theorems~{\rm\ref{bound_bdd}} and~{\rm\ref{bound_hold}} extends to the
      Heisenberg setting the results obtained by Korvenp\"a\"a, Kuusi  and one of the authors in~\cite{KKP16}, as well as they generalize to the fractional setting the regularity properties obtained for the obstacle problem in Carnot groups; see for instance~\cite{DGP07,DGS03} and the references therein. Nevertheless, the involved functional and geometrical settings are different than the ones considered in the aforementioned papers. Indeed, given the nonlinearity~$f \equiv f(\cdot,u)$, one has take into account the monotonicity assumption in~$(F_3)$ in order to extend the theory presented in~\cite{KKP16} to a more general nonhomogeneous framework. 

    Very general boundary regularity is in most cases a very challenging task in~$\He$. We would just mention the relevant  counterexample proven by Jerison in the seminal papers~\cite{Jer81,Jer81-2}. 
    For what concerns Theorem~\ref{bound_hold}, let us stress that we 
   proved the transferring of regularity from the obstacle (whether is not identically equals to~$-\infty$) and boundary datum to the solution~$u$ of the obstacle problem.
    In the same vein of further developments, it is quite natural to wonder whether or not the~H\"older-regularity obtained in Theorem~\ref{bound_hold} can be improved (under further assumptions on the data) at some extent: at least in the linear case, when~$p=2$, to the best of our knowledge no further regularity is known in the fractional case even though several breakthrough results have been achieved in the Euclidean case; see~\cite{CSS08,Sil07} for $C^{1,s}$-regularity estimates and related results.

	\vspace{2mm}
	%
	%
	\subsection{Boundedness and H\"older continuity up to the boundary for fractional Laplace problems}\label{sec_laplace}
\mbox{}	
\\ Let us come back to the $p$-fractional Laplace problem in the Heisenberg group; that is,
\begin{equation}\label{problema}
\begin{cases}
			\Lc u   = 0 & \text{in} \ \Omega,\\[0.5ex]
			u   =g & \text{in} \ \He \smallsetminus\Omega,
		\end{cases}
	\end{equation}
	where~$\Lc$ is given by~\eqref{main_op}, and  $g$ belongs to $HW^{s,p}(\He)$.

The boundedness and H\"older results in Section~\ref{sec_hold} can be generalized at some extents up to the boundary of~$\Omega$, by taking into account the measure density condition~\eqref{condizione misura}.
	\begin{thm}{\rm({\bf Boundedness up to the boundary}. See \cite[Theorem 3.6]{Pic22})}.
Let $s\in (0,1)$ and $p \in (1,\infty)$. Suppose that the domain~$\Omega\subset\He$ satisfies the measure density condition in~\eqref{condizione misura}. Let $x_0\in \partial \Omega$ and suppose that  the boundary data~$g$ is  essentially bounded close to $x_0$.
Let $u$ be a weak solution to~\eqref{problema}. Then $u$ is essentially bounded close to $x_0$ as well.
\end{thm}
The result above is merely an application of Theorem~\ref{bound_bdd} in the previous section for the fractional obstacle problem: it just suffices to consider the special case when there is no obstacle at all; precisely, take $\Omega'\equiv\He$ and $h\equiv -\infty$ there.
	
	The same strategy can be applied by using Theorem~\ref{bound_hold} in order to prove H\"older continuity results up to the boundary for the solutions to the $p$-fractional problem. We have the following
	\begin{thm}{\rm({\bf H\"older regularity up to the boundary}. See \cite[Theorem 1.3]{Pic22})}.
	Let $s\in (0,1)$ and $p \in (1,\infty)$. Suppose that the domain~$\Omega\subset\He$ satisfies the measure density condition in~\eqref{condizione misura}. Let $x_0\in \partial \Omega$ and suppose that  the boundary data $g$  is  H\"older continuous close to $x_0$.
Let $u$ be a weak solution to~\eqref{problema}. Then $u$ is H\"older continuous close to $x_0$ as well.
	\end{thm}
	
	\vspace{2mm}
	%
	%
	\section{The Perron method}\label{sec_perron}
		  In this section we investigate the Dirichlet problem~\eqref{pbm} in a very general setting by proving the existence of a generalized solution in the sense of Perron. As well known, Perron's Method is a classical tool in Potential Theory and it provides the existence of a solution for the Dirichlet problem related to the Laplace equation in an arbitrary open set~$\Omega \subset \er^n$ and for arbitrary boundary datum~$g$, without \ap regularity assumption". 
	  Roughly speaking, it simply works by finding the least superharmonic function greater or equal to the datum~$g$ on the boundary~$\partial \Omega$, and it can be applied once knowing some helpful properties of superharmonic functions, such as comparison and maximum principles, and barriers techniques.  
	  
	  For a deep treatment of the classical Potential Theory in stratified Lie groups, we refer the interested reader to the monograph~\cite{BLU07} by Bonfiglioli, Lanconelli and Uguzzoni.
	  
	\vspace{2mm}
    Let us focus on our fractional framework. The starting point is into giving a proper definition of upper / lower class of functions.
	
	\begin{defn}[{\bf Perron solutions}. See Definition 5.1~in~\cite{Pic22}]\label{perron classes and solutions} Let~$\Omega$ be an open and bounded subset of~$\He$ and assume that~$g \in HW^{s,p}(\He)$. We define the \textup{upper class} ~$\mathcal{U}_g$ as the family of all functions~$u$ such that
		\begin{enumerate}[\rm(i)]
			\item $u: \He \rightarrow (-\infty,+\infty]$ is lower semicontinuous in~$\Omega$ and
			there exists~$h \in L^{p-1}_{sp}(\He)$ such that~$u \leq h$;
			\item  Given~$D \Subset \Omega$ and $v \in C(\overline{D})$ a weak solution in~$D$ to~\eqref{pbm} such that~$v \leq u$ almost everywhere in $\He \smallsetminus D$ and
			$$
			v(\xi) \leq \liminf\limits_{\eta \rightarrow \xi} u(\eta), \qquad \forall \xi \in \partial D.
			$$ 
			Then,~$v \leq u$ a.\!~e. in~$D$ as well;
			\item  $\liminf\limits_{\eta \rightarrow \xi \atop \eta \in \Omega}u(\eta) \geq \textup{ess}\limsup\limits_{\eta \rightarrow \xi \atop \eta \in \He \smallsetminus\Omega}g(\eta)$ for all~$\xi \in \partial \Omega$;
			\item $u \geq  g $ a.\!~e. in~$\He \smallsetminus\Omega$.
		\end{enumerate}
		The \textup{lower class} is~$\mathcal{L}_g:=\big\{ u: -u \in \mathcal{U}_{-g}\}$. The function~$\overline{H}_g := \inf\{u : u \in \mathcal{U}_g\}$ is the \textup{upper Perron solution} with boundary datum~$g$ in~$\Omega$ and~$\underline{H}_g := \sup\{u : u \in \mathcal{L}_g\}$ is the \textup{lower Perron solution} with boundary datum~$g$ in~$\Omega$.
     	\end{defn}
	
	\begin{rem}
	Note that in view of~{\rm(i)} in~{\rm Definition~\ref{perron classes and solutions}}, one can deduce that $|u| \neq \infty$ a.~\!e.~in~$\He$, otherwise it can not be~$h \in L^{p-1}_{sp}(\He)$ such that~$u \leq h$ in~$\He$. Moreover, the upper class~$\mathcal{U}_g$ contains all lower semicontinuous weak supersolutions satisfying the boundary data. The same holds true for weak subsolutions and for the lower class~$\mathcal{L}_g$. 
	\end{rem}

Now, an expected comparison principle result.
  \begin{thm}[{\bf Comparison principle}. See Proposition~4.4 in~\cite{Pic22}]\label{comparison a.e.}
   Let~$\Omega$ be a bounded open subsets of~$\h^n$. Let~$u \in HW^{s,p}(\Omega)$ be a weak supersolution to~\eqref{pbm} in~$\Omega$ and let~$v\in HW^{s,p}(\Omega)$ be a weak subsolution to~\eqref{pbm} in~$\Omega$ such that~$ v \leq u$ almost everywhere in~$\h^n \smallsetminus \Omega$ and 
    $$
   	\limsup_{\eta \rightarrow \xi}v(\eta) \leq \liminf\limits_{\eta \rightarrow \xi} u(\eta), \qquad \mbox{for any}~\xi \in \partial \Omega.
   $$ 
   Then, under assumptions~$(F_1)$,~$(F_2)$ and~$(F_3)$, we have that~$v \leq u$ a.\!~e. in~$\Omega$ as well.
   \end{thm}
   \begin{proof}
Notice that there exists a compact set~$K$ such that~$\{v > u\} \Subset K \Subset \Omega$. For this, the function~$\psi := (v-u)_+ \in HW^{s,p}_0(K)$ is a proper test function. Considering the weak formulation  for both~$v$ and~$u$ in Definition~\ref{solution to inhomo pbm} with~$\psi$ defined above yields
  \begin{equation}\label{s3 comp 2}
  \int_{\h^n} f(\xi,v) \psi(\xi) \dxi \geq \int_{\h^n}\int_{\h^n} \frac{L(v(\xi),v(\eta))(\psi(\xi)-\psi(\eta))}{\dd(\eta^{-1} \circ \xi)^{Q+sp}}\dxieta,
  \end{equation}
  and
   \begin{equation}\label{s3 comp 3}
  	\int_{\h^n} f(\xi,u) \psi(\xi) \dxi \leq \int_{\h^n}\int_{\h^n} \frac{L(u(\xi),u(\eta))(\psi(\xi)-\psi(\eta))}{\dd(\eta^{-1} \circ \xi)^{Q+sp}}\dxieta,
  \end{equation}
  where, for the sake of readability, we adopted the following notation
   \begin{equation*} 
   L(a,b) := |a-b|^{p-2}(a-b) \qquad \forall \,a,b \in \mathds{R}.
   \end{equation*}

It suffices to sum up~\eqref{s3 comp 2} with~\eqref{s3 comp 3}, in order to get
   \begin{align}\label{s3 comp 1}
   0 & \leq \int_{\{v > u\}} (f(\xi,v)-f(\xi,u))(v(\xi)-u(\xi)) \dxi\notag\\
     & \quad + \iint_{\{v > u \}} \frac{\big(L(u(\xi),u(\eta))-L(v(\xi),v(\eta))\big)\big(v(\xi)-u(\xi)-v(\eta)+u(\eta)\big)} {\dd(\eta^{-1}\circ \xi)^{Q+sp}}\dxieta\notag\\
     & \quad + 2\int_{\{v \leq u\}}\int_{\{v > u \}} \frac{\big(L(u(\xi),u(\eta))-L(v(\xi),v(\eta)\big)\big(v(\xi)-u(\xi)\big)}{ \dd(\eta^{-1}\circ \xi)^{Q+sp}}\dxieta \ \ \leq 0,
   \end{align}
  by the monotonicity of the function~$L(a,b)$ and by condition~$(F_3)$.
  Thus, we obtain that all terms in~\eqref{s3 comp 1} must be~$0$. This implies that~$\psi \equiv0$ almost everywhere on~$\{v > u  \}$. Hence, we obtain that~$|\{v > u\}|=0$. 
  \end{proof}

     The following theorem states the existence of the generalized Perron solution  in the subRiemannian framework of the  Heisenberg group.
   \begin{thm}[{\bf Resolutivity}. See Theorem~1.4 in~\cite{Pic22}]\label{perron}
   	Let~$s \in (0,1)$ and~$p \in (1,\infty)$. Then, under assumptions~{\rm(}$F_1${\rm)},~{\rm(}$F_2${\rm)} and~{\rm(}$F_3${\rm)}, it holds that $\overline{H}_g = \underline{H}_g=H_g$, with~$\overline{H}_g,\underline{H}_g$ defined in Definition~{\rm \ref{perron classes and solutions}}.	
   	   	
   	Moreover,~$H_g$ is a continuous weak solution in~$\Omega$ to problem~\eqref{pbm}, according to Definition~{\rm \ref{solution to inhomo pbm}}.
   \end{thm}
    The proof of the results above are based on the introduction of the {\it Poisson modification} of a function~$u$ in a set~$D$. This is in clear accordance with the classical local framework. Clearly, most of the results for the fractional obstacle problem stated in Section~\ref{sec_obs} reveal to be necessary for the proper introduction of the related Perron method.
    \begin{defn}[{\bf Poisson modification of a continuous function}. See Definition~5.4 in~\cite{Pic22}]\label{poisson mod}
    	Let us consider a set~$D \Subset \Omega$ such that~$\He \smallsetminus D$ satisfies the measure density condition~\eqref{condizione misura} and let~$u \in C(\overline{D}) \cap L^{p-1}_{sp}(\He)$. We call \textup{Poisson modification} of~$u$ in the domain~$D$ the solution to the obstacle problem in~$D$ with boundary datum~$u$, i.~\!e.~$P_{u,D} \in \mathcal{K}_{u,-\infty}(D,\Omega) \cap C(\overline{D})$.
    \end{defn}
    
    By Corollary~\ref{properties_obst_sol}, we have that~$P_{u,D}$ is the weak solution of the following problem
    \begin{equation*}
    	\begin{cases}
    		\Lc w = f , \ & \mbox{in}\ D,\\
    		w= u,  \ & \mbox{in}\ \He \smallsetminus D.
    	\end{cases}
    \end{equation*}
    Once defined the Poisson modification for any continuous function, by a density argument, Definition~\ref{poisson mod} can be extended to all functions belonging to the upper class~$\mathcal{U}_g$; see~\cite[Proposition~5.5]{Pic22}.
    
   \vspace{2mm}
   For what concerns the {\it resolutivity result} in Theorem~\ref{perron}, it is worth noticing that one can not plainly apply the strategy in Euclidean counterpart studied in~\cite{KKP17}, where the class of~$(s,p)$-superharmonic functions have been introduced. Indeed, due to the presence of the nonlinearity term~$f \equiv f(\cdot,u)$ in the Dirichlet problem~\eqref{pbm}, most of the expected properties of $(s,p)$-superharmonic functions (which hold in the classical Euclidean and homogeneous case) are not immediately translated  to our more general nonhomogeneous Heisenberg setting. Hence, in~\cite{Pic22} it has been extended the strategy seen in~\cite{LL17} where the authors studied a non-homogeneous fractional problem. The dependence of the function~$f$ on the solution~$u$ is a novelty also with respect to~\cite{LL17} where the case when  $f=f(x)$ is considered. Such a difference does change drastically the background of the problem we are dealing with. For this, one has to take into account the assumption in~($F_2$) in order to prove some classical basic technique in Potential Theory as well as to apply the regularity results for weak solutions up to the boundary proven in the related obstacle problem.
Furthermore, one can prove that in the case when problem~\eqref{pbm} admits a weak solution~$h_g$ in~$\Omega$, it does coincide with the generalized Perron solution; see~\cite[Lemma~5.3]{Pic22}.
   
   \vspace{2mm}
   Once established the existence of the generalized Perron solution~$H_g$, it is quite natural to address the boundary regularity problem of where such a solution attains the boundary datum~$g$ on~$\partial \Omega$. The validity of a Wiener-type condition as for the fractional $p$-Laplace equation (see~\cite{KLL23}) is still an {open problem} in the Heisenberg framework.

	\section{Possible further developments}\label{sec_further}
   Despite the increasing interest in nonlocal equations whose underlying geometry is that of stratified Lie groups, there are still several questions which remain open. Below, we list just a few possible developments related to the results presented in the present notes.

    \vspace{2mm}
   	Firstly, it is worth remarking that in all the results stated in the preceding sections, we treat general weak solutions, namely by truncation and handling the resulting error term as a right-hand side in the spirit of De~Giorgi-Nash-Moser. 
   	However, one could deal with the same family of operators by focusing solely on bounded viscosity solutions in the spirit of Krylov-Safonov, as e.~\!g.~\cite{Sil06,KKL19}.   
    Consequently, a natural open problem is whether or not, and under which assumptions on the involved quantities, the viscosity solutions to nonlocal equations in~$\He$ are indeed weak solutions, and vice versa. 
    
    \vspace{2mm}
    Also, in the same spirit of the series of papers by Brasco, Lindgren, and Schikorra~\cite{BL17,BLS18}, one would expect higher differentiability. Similarly, one can investigate possible 
    self-improving properties of the solutions to~\eqref{pbm} as in the recent nonlocal Gehring-type theorems proven in~\cite{KMS15,Sch16} and~\cite{SM22}.
   
    \vspace{2mm}
    Furthermore, one could expect H\"older continuity and other regularity results for the solutions to a strictly related class of problems; that is, by adding in~\eqref{pbm} a  quasilinear operator modeled on the subelliptic $p$-Laplacian.  Apart from the recent paper~\cite{OT23}, the comprehension of mixed operators on stratified Lie groups is still not clearly understood as in the case of its Euclidean counterpart. Thus, regularity theory and related properties of weak solutions to~$\Lc u=0$ when~$\Lc$ is a mixed type operator are almost unknown. Clearly, one can expect that several results in the Euclidean setting, as e.~\!g., in~\cite{BMS23,DFM22,GK22,SZ23,MPP23,BLS23,OT23-2} do still hold for mixed-type operators in the Heisenberg group.
   
   \vspace{2mm}
   In a similar fashion of nonlocal operators, 
   one could expect nonlocal H\"arnack inequalities and other regularity results for the solutions to a strictly related class of problems; that is, by adding in~\eqref{pbm} a second integro-differential operator, $\mathcal{L}_{t,q;\alpha}$ of differentiability exponent~$t>s$ and summability growth~$q>1$, controlled by the zero set of a modulating coefficient~$\alpha\equiv\alpha(\xi,\eta)$; that is, the so-called nonlocal double phase problem, in the same spirit of the Euclidean case treated in~\cite{DFP19,BOS22}. In these direction very fundamental regularity results have been recently obtained both for double phase nonlocal equations in~$\He$ in the relevant paper~\cite{FZZ23}.
   
   \vspace{2mm}
   Moreover, we also remark that local boundedness, H\"older continuity and Harnack inequality have been recently obtained for nonlocal equations in~$\He$ with non-standard growth in~\cite{FZ23}, generalizing in the subRiemannian setting the estimates achieved in~\cite{CKW22,BKS23}.

    \vspace{2mm}
    Lastly,  it is natural to wonder whether or not the De Giorgi-Nash-Moser-type regularity results developed for weak solutions to~\eqref{pbm} can be generalized to the parabolic framework;       i.~\!e.~$\partial_t u(t,\xi) = \Lc u(t,\xi)$ in~$\He \times \er$, where the leading operator~$\Lc$ is defined by~\eqref{main_op}.
    In the fractional Euclidean setting, several remarkable achievements have been recently established in~\cite{Strom19,DZZ21} such as the boundedness estimates together with H\"older continuity for the superlinear case~$p\ge 2$, the latter generalized for any value of the integrability exponent in~\cite{Liao22,APT22}, and Harnack-type inequalities~\cite{Kim19,Strom19-2}.    
    In the fractional subRiemannian setting of the Heisenberg group to our knowledge almost all the natural questions are unsolved. Nevertheless, a first result establishing existence for solution to the related Cauchy-Dirichlet problem can be found in the recent paper~\cite{HS23}.

    \end{document}